\documentclass[12pt]{amsart}
\usepackage{geometry}   
\usepackage[colorlinks,citecolor = red, linkcolor=blue,hyperindex]{hyperref}
\usepackage{euscript,eufrak,verbatim, mathrsfs}
\usepackage[psamsfonts]{amssymb}
\usepackage{bbm}
\usepackage{graphicx}
 \usepackage{float}
\usepackage{float, tikz}

\usepackage{extpfeil}

 \usepackage[all, cmtip]{xy}

\usepackage{upref, xcolor, dsfont}
\usepackage{amsfonts,amsmath,amstext,amsbsy, amsopn,amsthm}
\usepackage{enumerate}

\usepackage{url}

\usepackage{mathtools}
\usepackage{bookmark}

 \usepackage{euscript}
\usepackage{mathptmx}       
\usepackage{helvet}         
\usepackage{courier}        
\usepackage{type1cm}        
\usepackage{multicol}        
\usepackage[bottom]{footmisc}

\newtheorem{theorem}{Theorem}[section]
\newtheorem*{theorem*}{Theorem B}
\newtheorem{lemma}[theorem]{Lemma}

\newtheorem{proposition}[theorem]{Proposition}

\newtheorem*{definition*}{Definition}
\newtheorem*{remark*}{Remark}

\newtheorem*{observation*}{Observation}

\newtheorem*{assumption*}{Assumption}
\newtheorem*{question*}{Question}
\newtheorem*{problem*}{Problem}

\newcommand{\R}{\mathbb{R}}

\newcommand{\E}{\mathbb{E}}

\newcommand{\PP}{\mathbb{P}}

\newcommand{\TT}{\mathcal{T}}

\newcommand{\CC}{\mathcal{C}}

\newcommand{\LL}{\mathcal{L}}

\newcommand{\FF}{\mathcal{F}}

\newcommand{\Var}{\mathrm{Var}}

\newcommand{\an}{\text{\, and \,}}

\begin{document}

\title{Moments of Mandelbrot cascades at critical exponents}

\author
{Yong Han}
\address
{Yong HAN: School of Mathematical Sciences, Shenzhen University, Shenzhen 518060, Guangdong, China}
\email{hanyong@szu.edu.cn}

\author
{Yanqi Qiu}
\address
{Yanqi Qiu: School of Fundamental Physics and Mathematical Sciences, HIAS, University of Chinese Academy of Sciences, Hangzhou 310024, China}
\email{yanqi.qiu@hotmail.com, yanqiqiu@ucas.ac.cn}

\author{Zipeng Wang}
\address{Zipeng WANG: College of Mathematics and Statistics, Chongqing University, Chongqing
401331, China}
\email{zipengwang2012@gmail.com}

\begin{abstract}
We obtain the asymptotic growth rate of the moments of the Mandelbrot random  cascades at critical exponents.  The key ingredient  is a $q$ to $q/2$ reduction method for the moment-estimation, which is obtained by combining  the martingale inequalities due to Burkholder  and   Burkholder-Rosenthal.   
\end{abstract}


\maketitle

\setcounter{equation}{0}

\section{Introduction}
Let us recall some basic results on the Mandelbrot random cascades  introduced in \cite{M74}.    For the rich litterature on  this topic, the reader is referred to \cite{Liu-spa, Fan-JMPA, Barral} and references therein.  

   We follow  the notation in Kahane-Peyri\`ere \cite{Kahane-Peyriere-advance} (see also its English translation \cite[pp. 372--388]{KP-en}).  Throughout the whole paper, we fix  an integer $b \ge 2$.  The $b$-adic intervals for $n =1, 2, \dots$ and $j_k = 0, \dots, b-1$ are defined by 
\[
I(j_1, j_2, \dots, j_n) = \Big[\sum_1^n j_k b^{-k},  \sum_1^n j_k b^{-k} + b^{-n} \Big[. 
\]
Consider a non-negative random variable $W$  with $\E[W]=1$ and let $W(j_1, j_2, \dots, j_n)$ denote the  independent copies of $W$. Let $\mu_n$ be the random measure defined on $[0,1]$ whose density  on the interval $I(j_1, j_2, \dots, j_n)$ is given by $W(j_1) W(j_1, j_2) \dots W(j_1, j_2, \dots, j_n)$.  

The main topic in Mandelbrot random cascades  is the study of limit behavior of  the random measures $\mu_n$ as $n$ goes to infinity. It turns out that  the limit behavior of $\mu_n$ is determined by  the non-negative martingale  
\begin{align}\label{def-Yn}
Y_n := \mu_n ([0,1])= b^{-n} \sum_{j_1, j_2 \dots, j_n} W(j_1) W(j_1,j_2)\dots W(j_1, j_2, \dots, j_n). 
\end{align}
Kahane-Peyri\`ere \cite{Kahane-Peyriere-advance} obtained the following fundamental results: 
\begin{itemize}
\item The random measures $\mu_n$ converge weakly to  a non-trivial random measure if and only if the martingale $(Y_n)_{n\ge 1}$ is uniformly integrable, and this is  in turn equivalent to  the condition $\E[W\log W] <\log b$.  
\item  Given  any  $q >1$.   The martingale $(Y_n)_{n\ge 1}$ is uniformly $L^q$-bounded if and only if $\E[W^q]<b^{q-1}$.  
\end{itemize}
In \cite{Kahane-Peyriere-advance}, the authors  associated with $W$ the convex function  
\[
\varphi_W(q)= \frac{\log \E[W^q]}{\log b} - (q-1) \in \R\cup\{+\infty\}, 
\]
which is always finite for $0 < q \le 1$ (since $W$ is integrable), and can be finite for  $q > 1$ if $\E[W^q]<\infty$. 
The function $\varphi_W$ is zero at the point $1$, and at most at one other point, $q_{crit}$, except for the following single case where $\varphi_W$ is identically zero (this case seems to be  ignored in \cite{Kahane-Peyriere-advance} since the main focus there is  on the subcritical exponents): 
\begin{align}\label{ex-W}
W_{TC} = \left\{ \begin{array}{cl} 
0 &  \text{with probability $1- b^{-1}$}
\\
b & \text{with probability $b^{-1}$}
\end{array}. 
\right. 
\end{align}
See Lemma \ref{lem-TC} for the proof of the following fact: 
\[
 \varphi_W\equiv 0   \Longleftrightarrow  \text{$\varphi_W$ vanishes at two points in $(1,\infty)$}  \Longleftrightarrow W \stackrel{d}{=} W_{TC}. 
 \]
For convenience, we say that  $W_{TC}$ is the totally critical random variable.

   The main contribution of this paper is the following refinement of the second part of   Kahane-Peyri\`ere's  result. We obtain the asymptotic  growth rate of $\E[Y_n^q]$ at its critical exponent $q= q_{crit}$. The situations  are dramastically different according to whether $W \stackrel{d}{=} W_{TC}$ (every exponent $q>1$ is critical for $W_{TC}$) or not.   

\begin{theorem}\label{thm-TC}
Let $W \stackrel{d}{=} W_{TC}$.  Then for any $q>1$, we have 
\[
\lim_{n\to \infty} \frac{\log \E[Y_n^q]}{\log n} = q-1.
\]
\end{theorem}

\begin{theorem}\label{thm-main}
Assume that $W \stackrel{d}{\ne} W_{TC}$.   Let $q\ge 2$. If $\E[W^q]= b^{q-1}$, then 
\begin{align*}
  \lim_{n\to \infty} \frac{\log \E [Y_n^q] }{\log n}=1. 
 \end{align*}
\end{theorem}

\begin{remark*}
In Theorem \ref{thm-main}, our method also yields 
\[
\limsup_{n\to\infty}    \frac{\log \E [Y_n^q] }{\log n}\le 1 \quad  \text{for \, } 1<q<2. 
\]
However, at the time of writing,  for $1<q<2$, we are not able to obtain 
\[
\liminf_{n\to\infty}    \frac{\log \E [Y_n^q] }{\log n}. 
\]
\end{remark*}

For subcritical and supercritical cases, we have 
\begin{proposition}\label{prop-main}
Assume that $W \stackrel{d}{\ne} W_{TC}$.   Let $q>1$. 
\begin{itemize}
 \item[(1)] Subcritical case: if $\E[W^q] < b^{q-1}$, then 
 \[
 \sup_{n}\E[Y_n^q]<\infty. 
 \]
 \item[(2)] Supercritical case: if $\E[W^q]> b^{q-1}$, then 
 \begin{align*}
 \liminf_{n\to\infty} \frac{\log \E[Y_n^q] }{n}\ge \log \frac{\E[W^q]}{b^{q-1}}
 \end{align*}   
  and if moreover $1<q\le 2$, then 
 \[
  \lim_{n\to\infty} \frac{\log \E[Y_n^q] }{n} =  \log \frac{\E[W^q]}{b^{q-1}}. 
 \]
 \end{itemize}
\end{proposition}

Our method is rather elementary by appropriately applying the standard Burkholder inequalities, as well as the Burkholder-Rosenthal inequalities in martingale theory.  The main steps are the following: 
\begin{itemize}
\item  We  introuduce a very general family of weighted sums of random variables on a given tree.  See the definition  \eqref{def-Theta}. 
\item Up to a  multiplicative constant,  we reduce  the $q$-moments of  a weighted sum of random variables on trees to the $\frac{q}{2}$-moments of another weighted  sum  with similar structure. See Lemma \ref{lem-red}.  
\item By repeating the $q$ to $q/2$ reduction method, we are led to compute the $p$-moment  (with $1<p<2$) of certain weighted sum of  random variables on a $b$-adic tree.  
\item In the totally critical case $W = W_{TC}$, we first obtain the growth rate of $\E[Y_n^q]$ for dyadic exponents $q= 2^\ell$. Then, using an  interpolation-extrapolation method, we obtain the growth rate of $\E[Y_n^q]$ for general $1<q<\infty$.  See \S \ref{sec-TC}. 
\item  In the case $W\stackrel{d}{\ne} W_{TC}$. We can not apply interpolation method and need to apply  a two-sided estimation (see Lemma~\ref{lem-lessthantwo}).  It turns out that, in our situation, the a priori different upper and lower estimates in Lemma~\ref{lem-lessthantwo} coincide on the level of asymptotic order.
\end{itemize}

Throughout the paper, by $A\lesssim_{x,y} B$, we mean there exists a constant $C_{x,y}>0$ depending only on $x,y$ such that $A \le C_{x,y}B$. Similarly, by  $A \asymp_{x,y} B$, we mean  $A\lesssim_{x,y}B$ and $B\lesssim_{x,y} A$.

{\flushleft\bf Acknowledgements.} This work is supported by the National Natural Science Foundation of China (No.12288201).  YH is supported by NSFC 12131016 and 12201419,  ZW is supported by NSF of Chongqing (CSTB2022BSXM-JCX0088,
CSTB2022NSCQ-MSX0321) and FRFCU (2023CDJXY-043).

\section{The $q$ to $q/2$ reduction method}
We mention that, even if $q_{crit}$ is an integer,  it is already quite complicated for  obtaining the asymptotic growth rate of $\E[Y_n^{q_{crit}}]$ for the random variables $Y_n$ defined in \eqref{def-Yn}.   For instance, the fourth moment of $Y_n$  involves the complicated general configurations of four vertices (not necessarily distinct) in a $b$-adic tree. 

Therefore, for proving Theorem \ref{thm-main} and Theorem \ref{thm-TC},  we introduce a $q$ to $q/2$ reduction method. This method works for more general weighted sum of random variables on trees and does not depends on the structure of the back-ground tree . In a sequel to this paper, this method will be applied to study random variables on the random Galton-Watson trees.

\subsection{Statement of the reduction method}

Let $\TT$ be a rooted tree  with root-vertex  denoted by $\rho$, equipped with the natural partial order $\preceq$ (making the root-vertex the smallest one) and the graph distance $d(\cdot,\cdot)$.


By a weight on $\TT$, we mean  a collection   $\alpha = \left\{\alpha(v): v\in \TT\right\}$  of  non-negative real numbers indexed by vertices of $\TT$.   Given any non-negative random variable $X$, not necessarily satisfying  the assumption $\E[X] = 1$, we let $\left\{X(v): v\in \TT\setminus\{\rho\}\right\}$ be a family of independent copies of $X$. For convenience,  set 
\[
X(\rho)\equiv 1.
\]  Define  a weighted sum as 
\begin{align}\label{def-Theta}
\Theta_{\TT}(X, \alpha): = \sum_{v\in \TT} \alpha(v)\prod_{u\preceq v} X(u). 
\end{align}

For any $v\in \TT$, denote  
\[
\TT(v)= \{u\in \TT: u \succeq v \} \an \mathcal{C}(v)= \{u\in \TT: u\succeq v \an  d(u,v)=1\}.
\]
That is,  $\TT(v)$ is  sub-tree of $\TT$ starting from $v$  and $\CC(v)$ is the set of all children of $v$.

\begin{lemma}\label{lem-red}
For  $q\geq 2$ and any  non-negative random variable $X$ with  $\E[X^q]<\infty$, we have
\begin{align}\label{eqn-main-inequality}
\mathbb{E}[\Theta_{\TT}(X, \alpha)^q] \asymp_{q, X} \E[\Theta_{\TT}(X^2,\beta)^{q/2}]
\end{align}
with $\beta: \TT\rightarrow \R_{+}$ being a new weight given by 
\begin{align}\label{eqn-new}
\beta(\rho)=\Big(\sum_{x\in \TT} \alpha(x)(\mathbb{E}[X])^{d(x, \rho)}\Big)^2+ \Var(X) \sum_{x\in \mathcal{C}(\rho)}\Big(\sum_{y\in \TT(x)}\alpha(y)   (\mathbb{E}[X])^{d(x,y)} \Big)^2
\end{align}
and for  any $v\ne \rho$,  
\begin{align}\label{eqn-new-weight}
\beta(v)= \Var(X) \sum_{x\in  \mathcal{C}(v)}\Big(\sum_{y\in \TT(x)}\alpha(y)  (\E[X])^{d(x,y)}  \Big)^2. 
\end{align}
The constants in \eqref{eqn-main-inequality} depend only on $q$ and  the quantity $\| X - \E[X]\|_q/\sqrt{\Var(X)}$.   More precisely,  there exist constants $c_q,  C_q$ depending only on $q$ such that 
\[
c_q  \cdot \E[  \Theta_\TT(X^2, \beta) ^{q/2}] \le \E[ \Theta_\TT (X, \alpha)^q ]   \le   C_q\frac{\| X- \E[X] \|_q^q}{\Var(X)^{q/2}}  \cdot \E[  \Theta_\TT(X^2, \beta) ^{q/2}]. 
\]
\end{lemma}

For $1\le q\le 2$,  we obtain in Lemma \ref{lem-lessthantwo} the  two-sided moment-estimates  for $\Theta_\TT(X,\alpha)$.

\begin{lemma}\label{lem-lessthantwo}
 For any $1\leq q \leq 2$, we have
 \begin{align*}
 \sum_{v\in \TT} (\mathbb{E}[X^q])^{d(v,\rho)} \alpha(v)^q \le \mathbb{E}[\Theta_{\TT}(X, \alpha)^q]
\le C_q \sum_{v\in \TT}(\E[X^q])^{d(v,\rho)}  \kappa(v)^q,
\end{align*}
where $C_q>0$ is a constant depending only on $q$ and $\kappa: \TT \rightarrow \R_{+}$  (which depends on $\alpha, X, \TT$) is defined by :  
\begin{align}\label{def-kappa-tt}
\kappa(v) = \kappa_{\alpha, X, \TT}(v): =  \sum_{y\in \TT(v)} \left(\mathbb{E}[X]\right)^{d(v,y)}\alpha(y), \quad \forall v  \in \TT. 
\end{align}
\end{lemma}
\begin{remark*}
Indeed, it is easy to see that the lower-estimates holds for all $q\ge 1$: 
\begin{align}\label{l-es}
 \sum_{v\in \TT} (\mathbb{E}[X^q])^{d(v,\rho)} \alpha(v)^q \le \mathbb{E}[\Theta_{\TT}(X, \alpha)^q]. 
\end{align}
\end{remark*}

\subsection{The proof of Lemma \ref{lem-red}}\label{sec-red}
For simplicity,  write
\[
M=\Theta_{\TT}(X,\alpha).
\]
We are going to  apply the Burkholder  and the Burkholder-Rosenthal inequalities. For this purpose,  we associate $M$ to a natural martingale as follows.  First, let $(\FF_m)_{m\ge 0}$ be the filtration of $\sigma$-algebras given  by 
\[
\FF_m : = \sigma \Big(\Big\{ X(v): v\in \TT \, \text{with}\, d(v, \rho)\le m\Big\}\Big), \quad m \ge 0.
\]
Define a martingale $(M_m)_{m\ge 0}$ by 
\begin{align*}
 M_m=\mathbb{E}[M|\mathcal{F}_m], \quad   m \ge 0. 
 \end{align*} 
The martingale increments for $(M_m)_{m\ge  0}$ are given  by
\[
D_0(M): =M_0=\mathbb{E}[M] \an 
D_m(M): =  M_m-M_{m-1} \, \text{for $m\ge 1$}.
\]
And the conditional square function is defined as
\begin{align}\label{def-sm}
s(M)^2: =M_0^2+\sum_{m\ge 1}\E[D_m(M)^2|\mathcal{F}_{m-1}]. 
\end{align}

 Recall the definition  \eqref{def-Theta} for weighted sums.  
\begin{lemma}\label{prop-alpha-beta}
The conditional square function \eqref{def-sm} is again a weighted sum on $\TT$:
\begin{align}\label{eqn-conditionl-square}
s(\Theta_\TT(X, \alpha))^2  = s(M)^2 = \Theta_\TT(X^2, \beta),
\end{align}
where  $\beta: \TT\rightarrow\R_{+}$ is a new weight modified from the original weight $\alpha$ and  given explicitly by the formulas \eqref{eqn-new} and \eqref{eqn-new-weight}.
\end{lemma}

To write more explicitly the sequence of martingale increments, we introduce the following notations.  For any $n\geq 0$, set \[
\mathcal{L}_n:=\{v\in\TT: d(v,\rho)=n\}. 
\]
 Moreover,  for  any $v \in \mathcal{L}_n$ and $m \le n$,    let $v_m$ be the unique vertex with 
\begin{align}\label{def-vm}
v \in \TT(v_m)  \an v_m \in \LL_m. 
\end{align}
 In the usual language in graph theory,  the above $\mathcal{L}_n$  is the sphere with radius  $n$ from the root $\rho$ and  $v_m$ is the unique ancestor of $v$ in $\mathcal{L}_m$. 

Using the above notation and writing
\begin{align}\label{def-M}
M =  \sum_{n\ge0} \sum_{v\in \mathcal{L}_n} \alpha(v)\prod_{u\preceq v} X(u), 
\end{align}
we obtain  
 \begin{align*} M_m=\sum_{n=0}^{m}\sum_{v\in \mathcal{L}_n}\alpha(v)\prod_{u\preceq v}X(u)+\sum_{n\ge m+1}\sum_{v\in \mathcal{L}_n}\alpha(v)\left(\mathbb{E}[X]\right)^{n-m}\prod_{u\preceq v_m} X(u). 
 \end{align*} 
Then it can be easily seen that the martingale increments for $(M_m)_{m\ge  0}$ are given by
\[
M_0=\sum_{n\ge 0} \sum_{v\in \LL_n} \alpha(v)(\E[X])^n
\]
and for $m\ge 1$, 
\begin{align}\label{D-m-ori-ex}
D_m(M)=  \sum_{n\ge m}\sum_{v\in \LL_n}\alpha(v)\left(\mathbb{E}[X]\right)^{n-m}  \mathring{X}(v_m) \prod_{u\prec v_{m}} X(u),
\end{align}
     where $\mathring{X}(v_m)$ is the centered random variable defined as 
  \[
 \mathring{X}(v_m): =  X(v_m) - \E[X(v_m)] = X(v_m) - \E[X]. 
  \]

 In Lemma  \ref{eqn-martingale-increment} below, we give an elementary but  useful   expression \eqref{eqn-martingale-increment} for $D_m(M)$. It is just a rearrangement of the terms in the sum of $D_m(M)$. 
  \begin{lemma}\label{lem-Dm-exp}
  For any $m \ge 1$, the martingale increment $D_m(M)$ has the expression: 
\begin{align}\label{eqn-martingale-increment}
D_m(M) =  \sum_{v\in \LL_m}  \Big(\mathring{X}(v)  \prod_{u\prec v} X(u)\Big)\sum_{n \ge m} \left(\mathbb{E}[X]\right)^{n-m}\sum_{w\in \LL_n \cap \TT(v)}\alpha(w).
\end{align}
  \end{lemma}
  \begin{proof}
 First write 
 \[
F(v): =  \mathring{X}(v)  \prod_{u\prec v} X(u) \an \LL_{\ge m}: = \bigcup_{n\ge m} \LL_n. 
 \]
Then  by \eqref{D-m-ori-ex}, 
  \begin{align*}
D_m(M)   =&    \sum_{n\ge m}\sum_{v\in \LL_n}\alpha(v)\left(\mathbb{E}[X]\right)^{n-m} F(v_m) =    \sum_{v\in  \LL_{\ge m}}\alpha(v)\left(\mathbb{E}[X]\right)^{d(v,\rho)-m}  F(v_m) 
  \\
  = & \sum_{x \in \LL_m}  F(x) \sum_{v\in \LL_{\ge m}, \, \, v_m = x }  \alpha(v)\left(\mathbb{E}[X]\right)^{d(v,\rho)-m}   
  \\
  = & \sum_{x \in \LL_m}  F(x)  \sum_{n\ge m}  \sum_{v \in \LL_n \cap \TT(x) } \alpha(v)\left(\mathbb{E}[X]\right)^{n-m} . 
  \end{align*}
  This completes the proof of the lemma. 
\end{proof}

Recall the definition of the function $\kappa$ defined in \eqref{def-kappa-tt}. We re-write $\kappa$ in the following way: for any $m \ge 1$ and any $x\in \LL_m$, 
\begin{align}\label{def-kappav}
\kappa(x): = \sum_{n\ge m} \left(\mathbb{E}[X]\right)^{n-m}\sum_{y\in \LL_n \cap \TT(x)}\alpha(y). 
\end{align}

\begin{lemma}\label{lem-cond-sq}
The conditional square function $s(M)$ satisfies the following equality:  
\begin{align*}
s(M)^2=&\Big(\sum_{x\in \TT} \alpha(x)(\E[X])^{d(x,\rho)}\Big)^2+\Var(X) \sum_{x\in \CC(\rho)}\Big(\sum_{n\ge 1} (\mathbb{E}[X])^{n-1}\sum_{y  \in \LL_n \cap \TT(x) }\alpha(y)\Big)^2\\
& + \Var(X)\sum_{m\ge2}\sum_{v\in \LL_{m-1}} \prod_{u\preceq v}X(u)^2\sum_{x\in  \CC(v)}\Big(\sum_{n\ge m} \left(\mathbb{E}[X]\right)^{n-m}\sum_{y\in \LL_n \cap \TT(x)}\alpha(y)\Big)^2.
\end{align*}
\end{lemma}

\begin{proof}
If $m=0$, then
\[
M_0 = \E[M] = \sum_{x\in \TT} \alpha(x)(\E[X])^{d(x,\rho)}.
\]
Now assume $m\ge 1$.  Applying the orthognality: for any $\FF_{m-1}$-measurable random variable $Z$, 
\[
\E [\mathring{X}(x) \mathring{X}(x') Z |\FF_{m-1}] = 0\quad \text{for any distinct\,} x, x'\in \LL_m,
\]
 we obtain  
\begin{align}\label{Dm-cond}
\begin{split}
\E[D_m(M)^2|\FF_{m-1}] = &  \E\Big[   \Big( \sum_{x\in \LL_m} \kappa(x)   \mathring{X}(x)  \prod_{u\prec x} X(u)  \Big)^2 \Big|\FF_{m-1}\Big]
\\
= &   \sum_{x\in \LL_m}  \E\Big[     \kappa(x)^2 \mathring{X}(x)^2  \prod_{u\prec x} X(u)^2  \Big|\FF_{m-1}\Big]
\\ 
= &  \Var(X) \sum_{x\in \LL_m}   \kappa(x)^2  \prod_{u\prec x} X(u)^2.
\end{split}
\end{align}
Recall the notation \eqref{def-vm}.  Note that,  for any $x \in \LL_m$,  
\[
u\prec x  \text{\, if and only if \,} u \preceq x_{m-1}. 
\]
Hence 
\begin{align*}
\E[D_m(M)^2|\FF_{m-1}] = &  \Var(X) \sum_{x\in \LL_m}   \kappa(x)^2  \prod_{u\preceq x_{m-1}} X(u)^2 
\\
  = &   \Var(X)  \sum_{v\in \LL_{m-1}} \sum_{x\in \LL_m \atop x_{m-1} = v}   \kappa(x)^2  \prod_{u\preceq v} X(u)^2 
  \\
  = & \Var(X)  \sum_{v\in \LL_{m-1}}  \sum_{x\in \CC(v)}   \kappa(x)^2 \prod_{u\preceq v} X(u)^2 .
\end{align*}
This completes the proof of the lemma. 
\end{proof}

\begin{proof}[Proof of Lemma \ref{prop-alpha-beta}]
By Lemma \ref{lem-cond-sq}, 
\begin{align*}
s(M)^2=&\underbrace{\Big(\sum_{x\in \TT} \alpha(x)(\E[X])^{d(x,\rho)}\Big)^2+\Var(X) \sum_{x\in \CC(\rho)}\kappa(x)^2}_{\text{this term coincides with $\beta(\rho)$}}\\
& + \sum_{m\ge2} \sum_{v\in \LL_{m-1}}  \Big(\prod_{u\preceq v}X(u)^2\Big) \underbrace{ \Var(X) \sum_{x\in  \CC(v)}\kappa(x)^2}_{\text{this term coincides with $\beta(v)$}}
\\
 = & \beta(\rho) +  \sum_{v\in \TT\setminus \rho}  \beta(v) \prod_{u\preceq v}X(u)^2 
 \\
 = & \Theta_\TT(X^2, \beta).
\end{align*}
This completes the proof of Lemma \ref{prop-alpha-beta}. 
\end{proof}

For any $q \ge 2$, the Burkholder-Rosenthal inequality (see, e.g., \cite[Theorem 5.52, p. 192]{Martingales-Pisier-book}) says that 
\begin{align}\label{ineq-BR}
\|M\|_q  \asymp_q    \Big(\sum_{m\ge 0} \E[|D_{m}(M)|^{q}] \Big)^{1/q}+ \| s(M)\|_q.
\end{align}

 \begin{lemma}\label{lem-control} For any $q\geq 2$, 
\begin{align}\label{BR-diag}
 \Big(\sum_{m\ge 0} \E[|D_{m}(M)|^{q}] \Big)^{1/q}  \leq C_q \frac{\| \mathring{X}\|_q}{\sqrt{\Var(X)}}  \cdot \|s(M)\|_q, 
 \end{align}
  where $C_q$ is a numerical constant  depending only  on $q$ (in fact depending only on the constant appears in the $L^q$-version Burkholder inequality).
\end{lemma}

In the proof of the upper estimate \eqref{BR-diag}, we need to use the classical Burkholder's inequalities   (see, e.g., \cite[formula (5.8) in p. 153]{Martingales-Pisier-book}), namely, for any $1<p<\infty$ and  any martingale $(\E[N| \FF_m])_{m\ge 0}$ in $L^p$, we have 
\begin{align}\label{B-ineq}
\| N\|_p \asymp_p    \Big\| \Big(\sum_{m\ge 0} D_m(N)^2\Big)^{1/2}\Big\|_p. 
\end{align}
Note also that, we shall only use  the one-sided estimate (follows the Burkholder-Davis-Gundy inequalities), which indeed holds for all $1\le p<\infty$ (see \cite{Davis-p-equal-one} for the case $p=1$):  
\begin{align}\label{BD-ineq}
\| N\|_p \lesssim_p    \Big\| \Big(\sum_{m\ge 0} D_m(N)^2\Big)^{1/2}\Big\|_p. 
\end{align}
The following classical inequalities (with $a\ge 1$ and $0<b\le 1$) will be useful  
\begin{align}\label{a-and-b}
\sum_{n} |x_n|^{a} \le \Big(\sum_{n} |x_n|\Big)^{a} \an \sum_{n} |x_n|^{b} \ge  \Big(\sum_{n} |x_n|\Big)^{b},\quad \forall x_n\in \R.
\end{align}

\begin{proof}[Proof of Lemma \ref{lem-control}]
Recall the expression  \eqref{def-kappav} for $\kappa(x)$.  Take any $m \ge1$.  Using \eqref{eqn-martingale-increment} and Burkholder's inequality \eqref{B-ineq},  and using the independence and orthogonality  of all $\mathring{X}(v)$ for $v\in \LL_m$,  we have
\begin{align}\label{eqn-complex-two}
\begin{split}
 \mathbb{E}[|D_m(M)|^q|\mathcal{F}_{m-1}]  & =\mathbb{E}\Big[\Big|  \sum_{v\in \LL_m}  \mathring{X}(v)    \kappa(v) \prod_{u\prec v} X(u) \Big|^q \Big|\mathcal{F}_{m-1} \Big] \\ & 
  \lesssim_q  \,  \mathbb{E}\Big[\Big|\sum_{v\in \LL_m}\mathring{X}(v)^2  \kappa(v)^2 \prod_{u\prec v} X(u)^2 \Big|^{q/2} \Big|\mathcal{F}_{m-1}\Big].
  \end{split}
 \end{align} 
   Since $q\ge 2$, we may apply  the triangle inequality in $L^{q/2}$ and get 
    \begin{align*}
    &\Big(  \mathbb{E}\Big[\Big|\sum_{v\in \LL_m}\mathring{X}(v)^2   \kappa(v)^2   \prod_{u\prec v} X(u)^2  \Big|^{q/2} \Big|\mathcal{F}_{m-1}\Big]   \Big)^{2/q}\\
    \le & \sum_{v\in \LL_m}   \Big(\mathbb{E}  \Big[   |\mathring{X}(v)|^q    \kappa(v)^q   \prod_{u\prec v} X(u)^q  \Big|\mathcal{F}_{m-1}\Big]\Big)^{2/q}
    \\
    = & \sum_{v\in \LL_m}   \Big(\mathbb{E}  [   |\mathring{X}(v)|^q]\Big)^{2/q} \kappa(v)^2 \prod_{u\prec v} X(u)^2   
    \\
 = &      \| \mathring{X}\|_q^2  \sum_{v\in \LL_m}   \kappa(v)^2   \prod_{u\prec v} X(u)^2. 
    \end{align*}
Therefore,
\[
\mathbb{E}[|D_m(M)|^q]\lesssim_q \,     \| \mathring{X}\|_q^q  \cdot \E \Big[ \Big(    \sum_{v\in \LL_m}    \kappa(v)^2   \prod_{u\prec v} X(u)^2 \Big)^{q/2} \Big]. 
\]
It follows that (by \eqref{a-and-b}), 
\begin{align}\label{eqn-complex-five}
\begin{split}
\sum_{m\ge 0} \mathbb{E}\left[\left|D_m(M)\right|^q\right]   & \lesssim_q\,   (\E[M])^q +  \| \mathring{X}\|_q^q  \cdot \E \Big[    \sum_{m\ge 1} \Big(    \sum_{v\in \LL_m}   \kappa(v)^2  \prod_{u\prec v} X(u)^2 \Big)^{q/2} \Big] \\
\\
& \le (\E[M])^q    + \| \mathring{X}\|_q^q  \cdot \E \Big[   \Big(   \sum_{m\ge 1}    \sum_{v\in \LL_m}   \kappa(v)^2 \prod_{u\prec v} X(u)^2  \Big)^{q/2} \Big].
\end{split}
\end{align}
On the other hand, by \eqref{Dm-cond}, 
\begin{align}\label{sm-qnorm}
\begin{split}
\|s(M)\|_q^q& =     \E\Big [ \Big(   (\E[M])^2 + \Var(X)  \sum_{m\ge 1}\sum_{x\in \LL_m}   \kappa(x)^2  \prod_{u\prec x} X(u)^2  \Big)^{q/2} \Big]
\\
& \ge    (\E[M])^q +     \Var(X)^{q/2}   \cdot \E\Big [ \Big(    \sum_{m\ge 1}\sum_{x\in \LL_m}   \kappa(x)^2  \prod_{u\prec x} X(u)^2  \Big)^{q/2} \Big].
\end{split}
\end{align}
Comparing \eqref{eqn-complex-five} and \eqref{sm-qnorm} and using the inequality $\| \mathring{X}\|_q \ge \sqrt{\Var(X)}$ for $q\ge 2$, we complete the whole proof of the lemma.
\end{proof}

\begin{proof}[Proof of Lemma \ref{lem-red}]
Combining \eqref{ineq-BR}  and \eqref{BR-diag}, we have 
\[
c_q \| s(M)\|_q \le \| M\|_q \le   c_q' \Big(1+ C_q \frac{\| \mathring{X}\|_q}{\sqrt{\Var(X)}} \Big) \cdot \|s(M)\|_q, 
\]
with $c_q, c_q'$ depend only on $L^q$-version Burkholder-Rosenthal inequality and $C_q$ depends only on the $L^q$-version  Burkholder inequality.  Finally, by using Lemma \ref{prop-alpha-beta}, we complete the whole proof of Lemma \ref{lem-red}. 
\end{proof}

\subsection{The proof of Lemma \ref{lem-lessthantwo}}
Using the notation introduced in \S \ref{sec-red}, we need to show 
\begin{align*}
 \sum_{m\ge 0} (\mathbb{E}[X^q])^m\sum_{v\in \LL_m} \alpha(v)^q \le \mathbb{E}[M^q]
\le C_q \sum_{m\ge 0}(\E[X^q])^{m}  \sum_{v\in \LL_m}    \kappa(v)^q. 
\end{align*}
By \eqref{def-M} and \eqref{a-and-b}, 
\begin{align*}
\mathbb{E}[M^q] &  = \E\Big[ \Big( \sum_{m\ge0} \sum_{v\in \mathcal{L}_m} \alpha(v)\prod_{u\preceq v} X(u)\Big)^q \Big] 
\\
& \ge    \E\Big[ \sum_{m\ge0} \sum_{v\in \mathcal{L}_m} \alpha(v)^q \prod_{u\preceq v} X(u)^q \Big] 
\\
& = \sum_{m\ge 0} (\mathbb{E}[X^q])^m\sum_{v\in \LL_m} \alpha(v)^q.
\end{align*}
Now by \eqref{BD-ineq} and \eqref{a-and-b}, for any $1\le q\le 2$, 
\begin{align*}
\mathbb{E}[M^q]   & = \|M\|_q^q \lesssim_q\,  \Big\|\Big(\sum_{m\ge 0} D_m(M)^2   \Big)^{1/2}\Big\|_q^q
\\
& =\E\Big[\Big(\sum_{m\ge 0}D_m(M)^2\Big)^{q/2}\Big]
\\
&  \leq\sum_{m\ge 0}\mathbb{E}[|D_m(M)|^q].
\end{align*}
By  \eqref{eqn-complex-two} applied to all $1\le q<\infty$ \footnote{Note that in deriving \eqref{eqn-complex-two}, we only used the one sided estimate \eqref{BD-ineq} which holds for all $1\le q<\infty$.}  and by \eqref{a-and-b} again,  for any $m\ge 1$, 
 \begin{align*}
 \mathbb{E}[|D_m(M)|^q|\mathcal{F}_{m-1}]    
  & \lesssim_q  \,  \mathbb{E}\Big[\Big|\sum_{v\in \LL_m}\mathring{X}(v)^2  \kappa(v)^2  \prod_{u\prec v} X(u)^2\Big|^{q/2} \Big|\mathcal{F}_{m-1}\Big]
  \\
  & \le \mathbb{E}\Big[\sum_{v\in \LL_m} |\mathring{X}(v)|^q   \kappa(v)^q \prod_{u\prec v} X(u)^q \Big|\mathcal{F}_{m-1}\Big] 
  \\
  & =     \E[|\mathring{X}|^q]    \sum_{v\in \LL_m}  \kappa(v)^q \prod_{u\prec v} X(u)^q.
 \end{align*}
 Therefore,  combining  $\| X-\E[X]\|_q\le 2 \| X\|_q$,  we get
 \begin{align*}
  \mathbb{E}[|D_m(M)|^q]    &  \lesssim_q\,  \E[|\mathring{X}|^q]    \sum_{v\in \LL_m}  \kappa(v)^q  \cdot \E\Big[ \prod_{u\prec v} X(u)^q\Big]
  \\
  & =  \E[|\mathring{X}|^q]    \sum_{v\in \LL_m}  \kappa(v)^q    (\E[X^q])^{m-1}
  \\
  &    \lesssim_q\, \sum_{v\in \LL_m} \kappa(v)^q (\E[X^q])^m.
 \end{align*}
The desired inequality now follows  immediately (note that  $\E[M] = \kappa(\rho)$): 
 \begin{align*}
\mathbb{E}[M^q]   \lesssim_q\, (\E[M])^q +        \sum_{m\ge 1}  (\E[X^q])^{m} \sum_{v\in \LL_m}  \kappa(v)^q =  \sum_{m\ge 0}  (\E[X^q])^{m} \sum_{v\in \LL_m}  \kappa(v)^q .  
\end{align*}

\section{Proof of Main results}
 In this section, we shall still use the notations introduced in \S \ref{sec-red}.  
 \subsection{The random variable $Y_n$ viewed as weighted sums}
 First,  note that the random variable $Y_n$ defined in \eqref{def-Yn} is a weighted sum in the sense of   \eqref{def-Theta}.  Indeed,  let $\TT_b$ be the infinite rooted $b$-adic tree. For any $n\ge 1$,  consider the weight $\alpha_n: \TT_b\rightarrow \R_{+}$ defined by 
\begin{align}\label{def-an}
\alpha_n(v)=   b^{-n} \mathds{1}(v\in \LL_n) = b^{-n} \mathds{1}(d(v, \rho)= n), \quad \forall v\in \TT_b. 
\end{align}
Then, using the notation \eqref{def-Theta}, we have 
\begin{align}\label{Yn-WS}
Y_n =  \Theta_{\TT_b} (W, \alpha_n). 
\end{align}

\begin{lemma}\label{lem-one-step}
Let $q\ge 2$.  Then 
\[
\E[Y_n^q] \asymp_{q, b, W} \E\Big[ \Big(   \sum_{k = 0}^{n-1} \sum_{v\in \mathcal{L}_k}   b^{- 2k} \prod_{u\preceq v} W(u)^2   \Big)^{q/2}\Big]. 
\]
\end{lemma}

\begin{proof}
Using \eqref{def-an} and \eqref{Yn-WS}, by Lemma \ref{lem-red}, we have 
\[
\E[ Y_n^q] = \E[\Theta_{\TT_b}(W, \alpha_n)^q] \asymp_{q,W} \E[\Theta_{\TT_b}(W^2, \beta)^{q/2}],
\]
where  $\beta$ is computed by \eqref{eqn-new} and \eqref{eqn-new-weight}, and is given explicitly by 
\[
\beta(\rho)=  1 + \frac{\Var(W)}{b} \an 
\beta(v)= \frac{\Var(W)}{  b^{1+ 2d(v, \rho)}  }      \mathds{1}(d(v, \rho)< n) \quad \text{for $v\ne \rho$}. 
\]
Hence, 
\begin{align*}
\E[\Theta_{\TT_b}(W^2, \beta_n)^{q/2}]    & =  \E\Big[ \Big(  1 +   \Var(W)  \sum_{k = 0}^{n-1} \sum_{v\in \mathcal{L}_k}   b^{-1 - 2k} \prod_{u\preceq v} W(u)^2   \Big)^{q/2}\Big] 
\\ 
& \asymp_{q, b, W}    \E\Big[ \Big(   \sum_{k = 0}^{n-1} \sum_{v\in \mathcal{L}_k}   b^{- 2k} \prod_{u\preceq v} W(u)^2   \Big)^{q/2}\Big]. 
\end{align*}
This completes the proof of the lemma. 
\end{proof}

\subsection{Two elementary lemmas}

 Recall the random variable $W_{TC}$ defined in \eqref{ex-W}. 
\begin{lemma}\label{lem-TC}
Let  $W$ be a non-negative random variable with $\E[W]=1$.   Then the following assertions are equivalent: 
\begin{itemize}
\item[(i)]  $W\stackrel{d}{=} W_{TC}$;
\item[(ii)]  $\varphi_W \equiv 0$; 
\item[(iii)]    There exist two distinct $q_1 \ne q_2 \in (1, \infty)$ such that   $\varphi_W (q_1)= \varphi_W(q_2) = 0$. 
\end{itemize}
\end{lemma}
\begin{proof}
The implications $(i) \Rightarrow (ii) \Rightarrow(iii)$ are trivial.  We now show that $(iii)\Rightarrow (i)$.  First note that $\varphi_W(p_1) =\varphi_W(p_2) = 0$ implies 
\begin{align}\label{2-pt}
\E[W^{p_1}] = b^{p_1-1}  \an \E[W^{p_2}] = b^{p_2 - 1}. 
\end{align}
Since $\E[W]=1$, we may define a new probability measure $\mu= W\PP$.   Assume $p_1<p_2$.  The relations \eqref{2-pt}  combined with the  H\"older's inequality for the measure $\mu$ imply 
\[
b^{p_1-1}   =   \E[W^{p_1}]  = \int W^{p_1-1} d\mu \le \| W^{p_1-1}\|_{L^{\frac{p_2-1}{p_1-1}}(\mu)} \| 1\|_{L^{\frac{p_2-1}{p_2-p_1}} (\mu)} =  b^{p_1-1}. 
\]
The equality holds if and only if $W^{p_1-1}$ is parallel to the constant function with respect to $\mu$. In other words, 
\[
W = \lambda \mathds{1}_A. 
\]
Since $\E[W]=1$ and  $\E[W^{p_1}] = b^{p_1-1}$, we obtain $\lambda =  b$ and $\PP(A)= b^{-1}$. That is,  $W  \stackrel{d}{=} W_{TC}$. 
\end{proof}

\begin{lemma}\label{lem-strict-ineq}
Let $q>1$. Assume that  $W \stackrel{d}{\ne} W_{TC}$ is a non-negative random variable  with $\E[W]= 1$ and $\E[W^q] \le b^{q-1}$.      Then for any $p\in (1, q)$, we have the strict inequality 
\[
 \E[W^p] <b^{p-1}. 
\]
\end{lemma}
\begin{proof}
By the H\"older's inequality with respect to the measure $\mu:=W \PP$, we have 
\begin{align*}
\E[W^p] = \int W^{p-1}  d\mu    \le \| W^{p-1} \|_{L^\frac{q-1}{p-1}(\mu)} \| 1 \|_{L^\frac{q-1}{q-p}(\mu)} \le  b^{p-1}.
\end{align*}
 By Lemma \ref{lem-TC},  we complete the whole proof of the lemma. 
\end{proof}

\subsection{The proof of Theorem \ref{thm-TC}}\label{sec-TC}
Assume that $W \stackrel{d}{=}W_{TC}$.  We divide the proof of Theorem~\ref{thm-TC} into two cases: the dyadic exponents $q = 2^\ell$ and the general exponents $q > 1$. 
\subsubsection{Dyadic exponents} Assume first that  $q  = 2^{\ell}$ with $\ell \ge 1$ being an integer.  The following elementary observation will be repeatedly used: 
\[
\frac{(W_{TC})^2}{b} = W_{TC}. 
\]
Thus, for any $k\ge 0$ and any $v\in \LL_k$ (recall also that, by convention, we set $W_{TC}(\rho) \equiv 1$), 
\begin{align}\label{2-b-2}
b^{-k} \prod_{u\preceq v} W_{TC}(u)^2 = \prod_{u\preceq v} W_{TC}(u). 
\end{align}
This equality, combined with Lemma \ref{lem-one-step},  yields 
\begin{align}\label{1-red-to}
\E[Y_n^{2^\ell}] \asymp_{\ell, b} \E\Big[ \Big(  \sum_{k = 0}^{n-1} \sum_{v\in \mathcal{L}_k}   b^{- k} \prod_{u\preceq v} W_{TC}(u) \Big)^{2^{\ell-1}}\Big]. 
\end{align}
Now define, for any real number $x\ge 1$ (as usual, $\lceil x \rceil$ denotes  the smallest integer  $\ge x$), 
\[
S_x : = \sum_{k = 0}^{\lceil x \rceil} \sum_{v\in \mathcal{L}_k}   b^{- k} \prod_{u\preceq v} W_{TC}(u). 
\]
\begin{lemma}
For any $p \ge 2$, we have 
\begin{align}\label{n-S-n}
  n^{p} \cdot  \E[ (S_{n/2})^{p/2}] \lesssim_{p,b} \E[ S_n^{p}]  \lesssim_{p,b}   n^{p}   \cdot  \E[  S_n^{p/2}]. 
\end{align}
\end{lemma}
\begin{proof}
Consider the weight 
\[
\alpha(v)  : = b^{-d (v, \rho)}  \cdot \mathds{1}(d(v,\rho)\le n). 
\] 
Then $S_n  = \Theta_{\TT_b}(W_{TC}, \alpha)$.   
By Lemma \ref{lem-red}, we have 
\begin{align}\label{ha-ha-ha}
\E[S_n^p] \asymp_{p,b}  \E[\Theta_{\TT_b}(W_{TC}^2, \beta)^{p/2}],
\end{align}
where  $\beta$ is supported on $\cup_{m\le n-1} \LL_m$ and is  computed by \eqref{eqn-new} and \eqref{eqn-new-weight}. More precisely, 
\begin{align*}
\beta(\rho)=   \Big(\sum_{k=0}^{n}  1 \Big)^2 +  \Var(W_{TC})\cdot b \cdot \Big( \sum_{k = 1}^{n}  b^{-1}   \Big)^2  \asymp_{b} n^{2}
\end{align*}
and for any $v\in \LL_m$ with $1\le m   \le n-1$, 
\begin{align*}
\beta(v)& =  \Var(W_{TC}) \cdot b \cdot \Big(  \sum_{k  =m+1}^{n}  b^{-m-1}     \Big)^2 \asymp_{b} b^{-2m} (n-m)^2.
\end{align*}
Therefore, by \eqref{2-b-2} and  \eqref{ha-ha-ha}, 
\begin{align*}
\E[S_n^p]  & \asymp_{p, b}    \E \Big[ \Big( \sum_{k = 0}^{n-1}  \sum_{v\in \LL_k}   b^{-2k}  (n-k)^2 \prod_{u\preceq v} W_{TC}(u)^2   \Big)^{p/2}\Big] 
\\
& = \underbrace{ \E \Big[ \Big( \sum_{k = 0}^{n-1}  \sum_{v\in \LL_k}   b^{-k}  (n-k)^2 \prod_{u\preceq v} W_{TC}(u)   \Big)^{p/2}\Big]}_{\text{denoted $J_n$}}.  
\end{align*}
Clearly, on the one hand, 
\[
 J_n  \le    \E \Big[ \Big( \sum_{k = 0}^{n}  \sum_{v\in \LL_k}   b^{-k}  n^2 \prod_{u\preceq v} W_{TC}(u)   \Big)^{p/2}\Big] = n^p \cdot \E [S_n^{p/2}]. 
\]
On the other hand, 
\[
J_n\ge \E \Big[ \Big( \sum_{k = 0}^{\lceil n/2\rceil}  \sum_{v\in \LL_k}   b^{-k}  (n-k)^2 \prod_{u\preceq v} W_{TC}(u)   \Big)^{p/2}\Big] \gtrsim_p  n^p \cdot  [ (S_{n/2})^{p/2}].
\]
This completes the whole proof of the lemma. 
\end{proof}

\begin{proof}[Concluding the proof of Theorem \ref{thm-TC} for dyadic exponents]
Combining \eqref{1-red-to} and repeatly using the two-sided estimates \eqref{n-S-n}, we obtain 
\begin{align*}
 n^{2^{\ell-1} + 2^{\ell-2} + \cdot + 2} \cdot  \E[S_{n/2^{\ell-1}}]   \lesssim_{\ell,b}  \E[Y_{n+1}^{2^{\ell}}]  \asymp_{\ell, b} \E[S_n^{2^{\ell-1}}] \lesssim_{\ell, b}   n^{2^{\ell-1} + 2^{\ell-2} + \cdot + 2} \cdot  \E[S_n]. 
\end{align*}
Now the equalities $\E[S_n] = n+1$ and $\E[S_{n/2^{\ell-1}}] = \lceil \frac{n}{2^{\ell-1}} \rceil$ yield 
\[
\E[Y_{n+1}^{2^{\ell}}]  \asymp_{\ell, b}   n^{2^{\ell-1} + 2^{\ell-2} + \cdot + 2 +1} = n^{2^\ell - 1}. 
\]
This is the desired asymptotic order. 
\end{proof}

\subsubsection{General exponents}
For a general $q>1$,  we  use an interpolation-extrapolation method.  Indeed, take the smallest $\ell$  with $q < 2^{\ell}$ ($\ell$  depends on $q$). There exists $\theta \in (0, 1)$ with 
\[
\frac{1}{q} = \frac{1- \theta}{1} + \frac{\theta}{2^{\ell}}  \quad (\text{hence $1- 1/q =   (1 - 1/2^{\ell})\theta$)}. 
\] 
Then $\| Y_n\|_q \le \|Y_n\|_1^{1-\theta} \|Y_n\|_{2^\ell}^\theta$. Since $\|Y_n\|_1 = 1$, we have 
\begin{align*}
\|Y_n\|_q \le \|Y_n\|_{2^\ell}^\theta \lesssim_{q, b} (n^{(2^{\ell}-1)})^{\theta/2^{\ell}} =   n^{1- q^{-1}}. 
\end{align*}
On the other hand,  since $1< q < 2^{\ell} < 2^{\ell+1}$, there exists a unique $\theta'\in (0, 1)$ such that 
\[
\frac{1}{2^\ell} = \frac{1- \theta'}{q}  + \frac{\theta'}{2^{\ell+1}}. 
\]
Hence 
\begin{align*}
(n ^{2^\ell -1})^{1/2^\ell}  \lesssim_{q, b} \| Y_n\|_{2^{\ell}} \le \| Y_n\|_q^{1-\theta'} \|Y_n\|_{2^{\ell+1}}^{\theta'} \lesssim_{q,b}  \| Y_n\|_q^{1-\theta'}  (n^{2^{\ell+1} -1})^{\theta'/2^{\ell+1}}.
\end{align*}
Therefore, 
\[
\|Y_n\|_q \gtrsim_{q,b}  n^{(1- 1/2^\ell - \theta' + \theta' /2^{\ell+1})/(1-\theta')} = n^{1- (1/2^\ell - \theta'/2^{\ell+1} )/(1-\theta')} = n^{1- q^{-1}}
\]
and we  complete the whole proof of Theorem \ref{thm-TC}.

\subsection{The proof of Theorem \ref{thm-main}}
From now on, we shall always assume that 
\[
W\stackrel{d}{\ne} W_{TC}. 
\]
Inspired by Lemma \ref{lem-one-step}, for any integer $\ell\ge 1, n \ge 1$, we  introduce 
\begin{align}\label{def-Unl}
U_n (W, \ell): = \sum_{k=0}^{n} \sum_{v\in \LL_k} b^{-2^\ell  k }  \prod_{u \preceq v} W(u)^{2^\ell}. 
\end{align}

\begin{lemma}\label{lem-red-power}
Let   $p \ge 2$ and let $\ell\ge 1$ be an integer.   Assume that  $\E[W^{2^\ell}]<b^{2^\ell-1}$, then 
\begin{align}\label{up-down}
\E [ U_n(W,\ell)^p] \asymp_{p, b, \ell, W}   \E [ U_{n-1}(W,\ell+1)^{p/2}].
\end{align}
\end{lemma}
\begin{proof}
Consider the following weight 
\[
\alpha_{\ell, n}(v) = b^{-2^\ell d(v, \rho)} \mathds{1}(d(v,\rho)\le n), \quad v \in \TT_b. 
\]
Then, under the notation \eqref{def-Theta},  
\begin{align}\label{two-U}
U_n(W, \ell)= \Theta_{\TT_b}(W^{2^\ell}, \alpha_{\ell, n}) \an U_{n-1}(W, \ell+1)= \Theta_{\TT_b}(W^{2^\ell}, \alpha_{\ell-1, n+1}) . 
\end{align}
By Lemma \ref{lem-red}, 
\begin{align}\label{U-U}
\E[U_n(W, \ell)^{p}]  \asymp_{p, W, \ell}  \E[\Theta_{\TT_b} (W^{2^{\ell+1}}, \beta)^{p/2}],
\end{align}
where  $\beta$ is supported on $\cup_{m\le n-1} \LL_m$   and is given by 
\[
\beta(\rho)=  \Big(  \sum_{k = 0}^n  b^k b^{-2^\ell k}   (\mathbb{E}[W^{2^\ell}])^k \Big)^2+ \Var(W^{2^\ell})  \cdot b\cdot \Big( \sum_{k = 1}^n     b^{k-1}   b^{-2^\ell  k}  (\E[W^{2^\ell}])^{k-1}  \Big)^2
\]
and for $v\in \LL_m$ with $1\le m \le n-1$, 
\[
\beta(v) = \Var(W^{2^\ell})  \cdot b\cdot \Big( \sum_{k = m+1}^n     b^{k-m-1}   b^{-2^\ell  k}  (\E[W^{2^\ell}])^{k-m-1}  \Big)^2.
\]
Under the assumption  $\E[W^{2^\ell}]<b^{2^\ell-1}$,    for any $v\in \LL_m$ with $0\le m \le n-1$, 
\[
\beta(v) \asymp_{b, \ell, W}     b^{-2^{\ell+1}m}. 
\]
In other words, 
\begin{align}\label{beta-to-alpha}
\beta \asymp_{b, \ell, W} \alpha_{\ell+1, n-1}. 
\end{align}
The desired relation \eqref{up-down} follows from \eqref{two-U}, \eqref{U-U} and \eqref{beta-to-alpha}. 
\end{proof}

\begin{proof}[Proof of Theorem \ref{thm-main}]
Fix $q\ge 2$.  Assume that $W \stackrel{d}{\ne} W_{TC}$ and $\E[W^q]   =   b^{q-1}$.  Note that there exists a unique integer $\ell_0 \ge 1$ such that $2^{\ell_0} \le q <  2^{\ell_0 + 1} $.     Fix a large  integer $n$, say $n\ge  4 \ell_0$. 

Supose first that $\ell_0 = 1$, then by Lemma \ref{lem-one-step} and  \eqref{def-Unl}, 
\[
\E[Y_n^q] \asymp_{q,b,W} \E[U_{n-1}(W, 1)^{q/2}] = \E[U_{n-\ell_0}(W, \ell_0)^{q/2^{\ell_0}}]. 
\]
Now assume $\ell_0 \ge 2$. Then  by Lemma \ref{lem-strict-ineq}, 
\[
\E[W^{2^\ell}]<b^{2^\ell-1} \quad \text{for all\,}  1\le \ell \le \ell_0-1. 
\] 
Therefore, by Lemma \ref{lem-one-step} and repeatly using \eqref{up-down} in  Lemma \ref{lem-red-power}, we obtain 
\begin{align*}
 \E[Y_n^q] & \asymp_{q,b,W} \E[U_{n-1}(W, 1)^{q/2}]     \asymp_{q,b,W} \E[U_{n-2}(W, 2)^{q/2^2}]    
 \\
 & \asymp_{q,b,W}    \cdots   \asymp_{q,b,W}   \E[U_{n-\ell_0+1}(W, \ell_0-1)^{q/2^{\ell_0-1}}]  \asymp_{q,b,W} \E[U_{n-\ell_0}(W, \ell_0)^{q/2^{\ell_0}}].       
\end{align*}
Note that in the last asymptotic equivalence, we only need the condition  
\[
\E[W^{2^{\ell_0-1}}]<b^{2^{\ell_0-1} -1}.
\] 
In conclusion, 
\begin{align}\label{Yn-final}
 \E[Y_n^q] \asymp_{q, b, W}  \underbrace{ \E\Big[ \Big(   \sum_{k = 0}^{n-\ell_0} \sum_{v\in \mathcal{L}_k}   b^{ - 2^{\ell_0}k} \prod_{u\preceq v} W(u)^{2^{\ell_0}}  \Big)^{q/ 2^{\ell_0}}\Big]}_{\text{denoted $I_{\ell_0}$}}. 
\end{align}
Since $1\le q/2^{\ell_0} <2$,  by Lemma \ref{lem-lessthantwo} applied to the  triple $(\alpha'', W^{2^{\ell_0}}, \TT_b)$, with 
\[
\alpha''(v)= b^{-2^{\ell_0} d(v, \rho)} \mathds{1}(d(v, \rho)\le n-\ell_0),
\]
we obtain  for all $v\in \LL_m$ with $0 \le m \le n - \ell_0$, 
\[
\kappa_{\alpha'', W^{2^{\ell_0}}, \TT_b} (v)=  \sum_{k   = m}^{n-\ell_0} \left(\mathbb{E}[W^{2^{\ell_0}}]\right)^{k-m}  b^{-m} b^{-(2^{\ell_0}-1) k   } \le  \frac{b^{-2^{\ell_0}m}}{1 - \E[W^{2^{\ell_0}}]/b^{2^{\ell_0} -1}}
\]
and hence 
\begin{align}\label{2-side-es}
 \sum_{m=0}^{n-\ell_0} (\mathbb{E}[W^q])^m    b^{-m(q-1)}   \le I_{\ell_0} \le    \frac{C_q}{1 - \E[W^{2^{\ell_0}}]/b^{2^{\ell_0} -1}}   \sum_{m=0}^{n-\ell_0} (\mathbb{E}[W^q])^m    b^{-m(q-1)} . 
\end{align}
The assumption  $\E[W^q]   =  b^{q-1}$, the equivalence \eqref{Yn-final} and  the two-sided estimates  \eqref{2-side-es}  together yield 
\[
n \asymp_{q, b, W} \E[Y_n^q]. 
\] 
This completes the whole proof. 
\end{proof}

\subsection{The proof of Proposition \ref{prop-main}}

\subsubsection{The subcritical case} Indeed, if $q\ge 2$, then under the assumption $\E[W^q]< b^{q-1}$, we may repeat the arguments in the proof of Theorem \ref{thm-main} and arrive at the two-sided estimates \eqref{2-side-es}.  Then by $\E[W^q]< b^{q-1}$, we obtain 
$
\sup_n \E[Y_n^q]<\infty. 
$
Now assume that $1< q <2$ and $\E[W^q]<b^{q-1}$.  By Lemma \ref{lem-lessthantwo}, 
\begin{align*}
\E[Y_n^q] &  \lesssim_{q} \sum_{v\in \TT} (\E[W^q])^{d(v,\rho)} \Big(\sum_{y\in \TT:   \, v\preceq y}  b^{-n}\Big)^q 
\\
 & = \sum_{m  =0}^n  (\E[W^q])^m    \sum_{v\in \LL_m}   \Big(\sum_{y\in \TT:   \, v\preceq y}  b^{-n}\Big)^q 
 \\
 &  = \sum_{m =  0}^n  (\E[W^q])^m  \cdot  b^m  \cdot b^{-mq} \le \frac{1}{1 - \E[W^q]/b^{q-1}}. 
\end{align*}
The desired relation $\sup_{n} \E[Y_n^q]<\infty$ follows. 

\subsubsection{The supercritical case}

For any $q>1$,  by \eqref{l-es},  
\[
\E[Y_n^q]\ge     (\E[W^q])^n  \cdot b^{-n(q-1)}. 
\]
Thus 
\[
\liminf_{n\to\infty} \frac{\log \E[Y_n^q]}{n} \ge  \log\frac{\E[W^q]}{b^{q-1}}. 
\]
If moreover, $1<q\le 2$, then  by Lemma \ref{lem-lessthantwo},  we have 
\begin{align*}
\E[Y_n^q] &  \lesssim_{q} \sum_{v\in \TT} (\E[W^q])^{d(v,\rho)} \Big(\sum_{y\in \TT:   \, v\preceq y}  b^{-n}\Big)^q 
\\
 & = \sum_{m  =0}^n  (\E[W^q])^m    \sum_{v\in \LL_m}   \Big(\sum_{y\in \TT:   \, v\preceq y}  b^{-n}\Big)^q 
 \\
 &  = \sum_{m =  0}^n  (\E[W^q])^m  \cdot  b^m  \cdot b^{-mq} \le \frac{(\E[W^q]/b^{q-1})^n}{1 -  b^{q-1}/\E[W^q]}. 
\end{align*}
Therefore, 
\[
\limsup_{n\to\infty}  \frac{\log \E[Y_n^q]}{n} \le  \log\frac{\E[W^q]}{b^{q-1}}. 
\]
Hence
\[
\lim_{n\to\infty}   \frac{\log \E[Y_n^q]}{n}  =   \log\frac{\E[W^q]}{b^{q-1}} \quad \text{for $1<q\le 2$}. 
\]


\end{document}